\documentclass{article}
\usepackage{graphicx,color} 
\usepackage{amsmath,amssymb,amsthm,enumitem}
\usepackage[labelfont=bf,textfont=sf]{caption}
\newcommand{\N}{\mathbb{N}}
\newcommand{\Z}{\mathbb{Z}}

\newcommand{\editfull}{\mathrm{d}_E}

\newcommand{\bb}[1]{b_{\mu}(#1)}
\newcommand{\len}{\mathrm{Len}}
\newcommand{\maxgap}{\mathrm{maxgap}}
\newcommand{\mingap}{\mathrm{mingap}}

\newtheorem{introtheorem}{Theorem}[]
\newtheorem{introproposition}[introtheorem]{Proposition}
\newtheorem{theorem}{Theorem}[section]

\newtheorem{lemma}[theorem]{Lemma}

\theoremstyle{definition}
\newtheorem{remark}[theorem]{Remark}

\title{Edit distance in substitution systems}
\author{Andrew Best and Yuval Peres}
\date{\today}

\begin{document}
	\maketitle
	\begin{abstract}
		Let $\sigma$ be a primitive substitution on an alphabet $\mathcal{A}$, and let $\mathcal{W}_n$ be the set of words of length $n$ determined by $\sigma$ (i.e., $w \in \mathcal{W}_n$ if $w$ is a subword of $\sigma^k(a)$ for some $a \in \mathcal{A}$ and $k \geq 1$). It is known that the corresponding substitution dynamical system is loosely Kronecker (also known as zero-entropy loosely Bernoulli), so the diameter of $\mathcal{W}_n$ in the edit distance is $o(n)$. %
		We improve this upper bound to
		$O(n/\sqrt{\log n})$. The main challenge is handling the case where $\sigma$ is non-uniform; a better bound is available for the uniform case.  Finally, we show that for the Thue--Morse substitution, the diameter of $\mathcal{W}_n$ is at least $\sqrt {n/6} - 1$.         
	\end{abstract}
	
	\section{Introduction}
	
	Let $\mathcal{A}$ be a finite alphabet, and let $\mathcal{A}^+ = \cup_{\ell\geq 1} \mathcal{A}^\ell$ be the set of all \textbf{words} in $\mathcal{A}$. Given a map $\sigma : \mathcal{A} \to \mathcal{A}^+$, we extend $\sigma$ to $\mathcal{A}^+$ by defining, for all $k \in \N$, 
	\begin{equation}
		\forall  a_1,\ldots, a_k \in \mathcal{A}, \qquad \sigma(a_1 \cdots a_k) \ := \ \sigma(a_1) \cdots \sigma(a_k) \,.
	\end{equation}
	Such a map $\sigma$ is called a {\bf substitution}  on  $\mathcal A$. If there is an $\ell \in \N$ such that $ \sigma(\mathcal A) \subset  \mathcal{A}^\ell$, then $\sigma$ is said to be \textbf{$\ell$-uniform}, or simply \textbf{uniform}. A substitution $\sigma$ is \textbf{primitive} if there exists $k \in \N$ such that for all $m \geq k$, the word $\sigma^{m}(a)$ contains every letter in $\mathcal{A}$, where $\sigma^m$ denotes the $m$-fold iterated composition of $\sigma$. For example, the primitive 2-uniform substitution $\mu$ on $\{0,1\}$ defined by $\mu(0) = 01$ and $\mu(1) = 10$ is called the \textbf{Thue--Morse substitution}.
	
	The \textbf{edit distance} $\editfull(x,y)$ between two words $x,y \in \mathcal{A}^+$, is defined to be half the minimal number of operations required to transform $x$ into $y$, where an operation is either the deletion or the insertion of a single letter at an arbitrary position. For example, $\editfull(10002,03004) = 2$ and $\editfull(000,111) = 3$.  
	
	We say a word $u = u_1\ldots u_n$ is a \textbf{subword} of a word $w = w_1\ldots w_m$ and write $u \prec w$ if 
	there exists $s \geq 0$ such that $u_{i} = w_{i+s}$ for each $i$. For each $n \in \N$, define $\mathcal{W}_n(\sigma) $ to be the collection of all words $w \in \mathcal{A}^n$ such that $w \prec \sigma^k(a)$ for some $k \in \N$ and $a \in \mathcal{A}$.
	
	Given a substitution $\sigma$ on $\mathcal{A}$,
	denote
	\begin{equation}
		\Omega(\sigma) \ := \ \{ \omega \in \mathcal{A}^\Z : \text{Every subword of $\omega$ belongs to } \bigcup_{n=1}^\infty \mathcal{W}_n(\sigma)\}\,.
	\end{equation}
	Observe that $\Omega(\sigma)$ is invariant under the shift $T$ given by $(T\omega)(n) := \omega(n+1)$. 
	The pair $(\Omega(\sigma),T)$ is called a \textbf{substitution dynamical system}; it is minimal if and only if $\sigma$ is primitive; see, e.g., \cite[Proposition 5.5]{queffelec}.
	
	Feldman~\cite{feldman} defined loosely Bernoulli systems. Following Ratner \cite{ratner84}, we will refer to zero-entropy loosely Bernoulli systems as loosely Kronecker. Loosely Kronecker systems are precisely those which are Kakutani equivalent to irrational rotations on the torus.
	
	It is known that all minimal substitution dynamical systems are loosely Kronecker, and since they are uniquely ergodic, this implies that, for these systems,
	\begin{equation}\label{main theorem sloppy version} \mathrm{diam}_E(\mathcal{W}_n(\sigma)) \ = \ o(n).
	\end{equation}
	See Section~\ref{subsec: related results} for details. Our first result is a quantitative improvement of \eqref{main theorem sloppy version}.
	\begin{introtheorem}\label{main theorem}
		Let $\sigma$ be a primitive substitution where $|\mathcal{A}| > 1$. Then
		\begin{equation}\label{non-uniform bound on diameter}
			\mathrm{diam}_E(\mathcal{W}_n(\sigma)) \ := \ \max_{w,\tilde{w} \in \mathcal{W}_n(\sigma)} \editfull(w,\tilde{w}) \ = \ O\Bigl(\frac{n}{\sqrt{\log n}}\Bigr).
		\end{equation}
		If $\sigma$ is uniform, then \eqref{non-uniform bound on diameter} can be improved: there exists $\alpha = \alpha(\sigma) \in (0,1)$ such that
		\begin{equation}
			\mathrm{diam}_E(\mathcal{W}_n(\sigma)) \ = \ O\bigl(n\alpha^{\sqrt{\log n}}\bigr).
		\end{equation}
	\end{introtheorem}

\medskip

For uniform substitutions, we can relax the primitivity assumption and characterize when \eqref{main theorem sloppy version} holds.

Given a substitution $\sigma$, define the directed graph $G_\sigma$ with vertex set $\mathcal{A}$ and an edge from $a$ to $b$ if and only if $\sigma(a)$ contains the letter $b$. Thus, $b$ is reachable from $a$ in $G$ if and only if $b$ appears in $\sigma^n(a)$ for some $n \in \N$.

A strongly connected component $G_0$ of $G_\sigma$ is called \textbf{essential} if there are no edges from vertices in   $G_0$ to vertices in  another strongly connected component. For each $a \in \mathcal{A}$, let
$T(a) :=  \{ n \in \N : a \text{ appears in } \sigma^n(a) \}$ and define the \textbf{period} of $a$ as the greatest common divisor of  $T(a)$. All letters in the same strongly connected component have the same period. Define a strongly connected component to be \textbf{aperiodic} if the common period of its vertices is 1.

\begin{introproposition}\label{introprop: characterization of uniform case}
	Let $\sigma$ be an $\ell$-uniform substitution with $\ell > 1$.
	\begin{enumerate}[label=(\alph*)]
		\item If there is a unique essential strongly connected component $G_0$ of $G_\sigma$, and $G_0$ is aperiodic, then $\mathrm{diam}_E(\mathcal{W}_n(\sigma)) \ = \ o(n)$.
		\item Otherwise, $\mathrm{diam}_E(\mathcal{W}_n(\sigma)) \ = \ n$ for all $n \in \N$.
	\end{enumerate}
\end{introproposition}
\begin{remark}
	Proposition~\ref{introprop: characterization of uniform case} is similar to the ``zero-two theorem'' of \cite{os}. 
\end{remark}

For $w \in \{0,1\}^+$, let $\overline{w}$ be the bitwise complement of $w$, i.e., $\overline{w}_i=1-w_i$.

The Thue--Morse sequence $\mathbf{x}=(x_i)_{i=0}^\infty =(01101001\dots)$ is defined as follows. Set $x_0 = 0$, and for every $n \in \N \cup \{0\}$, let $x_{2n} = x_n$ and $x_{2n+1} = \overline{x}_n$.  
For each $n,m \in \N \cup \{0\}$ with $n < m$, write
\begin{align}
	x_{[n,m]} \ := \ x_{n}x_{n+1}\cdots x_{m} \quad \text{ and } \quad x_{[n,m)} \ := \ x_{n}x_{n+1}\cdots x_{m-1}.
\end{align}
Then $x_{[0,2^k)} = \mu^k(0)$ and $\overline{x}_{[0,2^k)} = \mu^k(1)$ for each $k \in \N$. See \cite{alloucheshallit} for a survey of many properties of this sequence.

Our final result is a lower bound for the edit distance: \begin{introproposition}\label{introprop: lower bound on edit distance for morse}
	For all $n \in \N$, we have $\editfull(x_{[0,n)},\overline{x_{[0,n)}}) \geq \sqrt{\frac{n}{6}} - 1$.
\end{introproposition}

\subsection{Related Results}\label{subsec: related results}
For words $x, y \in \mathcal{A}^+$ of respective length $n$ and $m$, we have the basic identity
\begin{equation}\label{eqn: lcs and edit distance relation} \editfull (x,y) \ := \ \frac{n+m}{2} - S(x,y),
\end{equation}
where $S(x,y)$ is the length of a longest common subsequence of $x$ and $y$. 

In 1976, Feldman \cite{feldman}  defined loosely Bernoulli systems in terms of edit distance, observed that irrational rotations on the circle are loosely Kronecker, and showed there exists a zero-entropy ergodic automorphism that is not loosely Kronecker. Katok~\cite{katok} in 1977 independently defined and characterized loosely Kronecker systems (which he called ``standard''). Subsequently, several classes of dynamical systems were shown to have this property.

Ratner \cite{ratner78} showed in 1978 that horocycle flows are loosely Kronecker. Ornstein, Rudolph, and Weiss~\cite[Theorem 7.3]{orw} showed that any ergodic skew product over a loosely Kronecker system with rotations of a compact group is also loosely Kronecker. 

In 1982, N\"{u}rnberg~\cite{nuernberg}  proved that the orbit closures of generalized Morse sequences (defined in 1968 by Keane~\cite{keane}) are loosely Kronecker. In fact, N\"{u}rnberg's argument implies that \eqref{main theorem sloppy version} holds for all uniform substitutions on two letters $\{0,1\}$ such that $\sigma(0) = w$ and $\sigma(1) = \overline{w}$ for some word $w \in \{0,1\}^+$.  

If $(\Omega,T)$ is a uniquely ergodic subshift that is loosely Kronecker, then \eqref{main theorem sloppy version} holds with $\mathcal{W}_n(\sigma)$ replaced by $\{ x_1\ldots x_n \in \mathcal{A}^n : (x_i)_{i\in\Z} \in \Omega \}$; see the proof of \cite[Proposition 6.10]{orw}. 

Here is a known proof of the fact that every minimal substitution dynamical system is loosely Kronecker. By Proposition 5.12 in \cite{queffelec}, the word complexity of such a system is at most linear. Thus, by Proposition 4 in Ferenczi \cite{ferenczi}, the system has finite rank. Finally, by Theorem 8.4 in \cite{orw}, every ergodic map with finite rank is loosely Bernoulli.

Recently, loosely Kronecker (and more generally, loosely Bernoulli) transformations have received renewed attention. Glasner, Thouvenot, and Weiss \cite{gtw} showed that, generically, an extension of an ergodic loosely Bernoulli system is loosely Bernoulli. Gerber and Kunde \cite{gerberkunde2020} showed there exist smooth weakly mixing loosely Kronecker transformations whose Cartesian square is loosely Kronecker, and 
Trujillo \cite{trujillo} later showed there exist such transformations that are mixing. Garc\'{i}a-Ramos and Kwietniak \cite{g-rk} characterized uniquely ergodic dynamical systems that are loosely Kronecker. Their result uses the Feldman--Katok pseudometric, introduced earlier by Kwietniak and Łacka in \cite{kl}.   
Trilles \cite{trilles} proved an analogue of this characterization for continuous flows.

In 2020, Blikstad \cite{blikstad} proved a quantitative upper bound on the edit distance between a prefix of the Thue--Morse sequence and its complement, namely   \begin{equation}\label{Blikstad bound on bb}
	\editfull(x_{[0,n)},\overline{x_{[0,n)}}) \ = \ 
	O\bigl(\sqrt{\log n} \cdot 2^{-\beta\sqrt{\log n}}\bigr),
\end{equation}
where $\beta = \sqrt{2\log_2(3)} \approx 1.78$.

\subsection{Notation}
Given a substitution $\sigma$ on $\mathcal{A}$,  the $|\mathcal{A}| \times |\mathcal{A}|$ \textbf{incidence matrix} $M_\sigma$ is defined by setting the $(a,b)$ entry $M_\sigma(a,b)$ to be the number of instances of the symbol $b$ in the word $\sigma(a)$. A square matrix $M$ with nonnegative real entries is called \textbf{primitive} if there exists $k \in \N$ such that every entry of $M^k$ is positive. Then $\sigma$ is primitive if and only if $M_{\sigma}$ is primitive.

Write $\len(w)=\ell$ for the \textbf{length} of $w \in \mathcal{A}^\ell$.

\subsection{Organization}
In Sections~\ref{sec: proof of main theorem in uniform case} and \ref{sec: proof of main theorem in non-uniform case}, we prove Theorem~\ref{main theorem}, first in the uniform case (where we obtain better quantitative bounds for edit distance) and then in the non-uniform case. In Section~\ref{sec: proof of characterization in uniform case}, we study uniform substitutions that are not necessarily primitive and prove Proposition~\ref{introprop: characterization of uniform case}. In Section~\ref{sec: proof of edit distance lower bound}, we focus on the Thue--Morse sequence and prove Proposition~\ref{introprop: lower bound on edit distance for morse}; we also state a related open problem.

\section{Proof of Theorem~\ref{main theorem} in the uniform case}\label{sec: proof of main theorem in uniform case}
A basic property of edit distance we will use is that
\begin{equation}
	\editfull(xy,\tilde{x}\tilde{y}) \ \leq \ \editfull(x,\tilde{x}) + \editfull(y,\tilde{y}) 
\end{equation}
for all words $x,y,\tilde{x},\tilde{y}$.

Let $\sigma$ be an $\ell$-uniform substitution with $\ell > 1$. For integer $k \geq 2$, define
\begin{equation}\label{definition of r_k}
	r_k \ := \ \max\left\{
	\frac{\editfull(w,\tilde{w})}{\len(w)} : w,\tilde{w} \in \mathcal{W}_n(\sigma) \text{ for some } n \in [\ell^{k^2},\ell^{(k+1)^2}) \right\}.
\end{equation}

\begin{lemma}\label{lem: recursive estimate on r_k}
	Let $\ell>1$ and fix an $\ell$-uniform substitution $\sigma$ with incidence matrix having a nonzero column. For each integer $k \geq 2$,
	\begin{equation}\label{eqn in lem: recursive estimate on r_k}
		r_{k+1} \ \leq \ \frac{1}{\ell^{k}} + \left(1 - \frac{1}{2\ell^{2}}\right)r_k.
	\end{equation}
\end{lemma}

\begin{proof}
	Let $k \geq 2$. Fix $n \in \N$ such that
	\begin{equation} \ell^{(k+1)^2} \ \leq \ n \ < \ \ell^{(k+2)^2},
	\end{equation}
	and fix $w,\tilde{w} \in \mathcal{W}_n(\sigma)$. Let $m,\tilde{m} \in \N$ and $\xi_0,\xi_1 \in \mathcal{A}$ be such that $w \prec \sigma^{m}(\xi_0)$ and $\tilde{w} \prec \sigma^{\tilde{m}}(\xi_1)$.
	
	Let $a_i \in \mathcal{A}$ be such that
	\begin{equation}
		\sigma^{m-k^2-1}(\xi_0) \ = \ a_1 \cdots a_{\ell^{m-k^2-1}}.
	\end{equation}
	Note that $\sigma^{m}(\xi_0) = \sigma^{k^2+1}(\sigma^{m-k^2-1}(\xi_0))$ is a concatenation of words $\sigma^{k^2+1}(a_i)$.
	Because $\sigma$ is $\ell$-uniform, $\len(\sigma^{k^2+1}(a_i)) = \ell^{k^2+1}$. Thus, since $w \prec \sigma^{m}(\xi_0)$ and
	\begin{equation}
		\len(w) \ \geq \ \ell^{(k+1)^2} \ \geq \ 2\ell^{k^2+1} - 1,
	\end{equation}
	$w$ contains at least one word $\sigma^{k^2+1}(a_i)$ as a subword. Hence, we may choose a maximal length subword $y$ of $\sigma^{m-k^2-1}(\xi_0)$ such that $\sigma^{k^2+1}(y) \prec w$. Note that
	\begin{equation}\label{inequality for q_0}
		\frac{n}{\ell^{k^2+1}} - 2 \ \leq \ \len(y) \ \leq \ \frac{n}{\ell^{k^2+1}}.
	\end{equation}
	
	Similarly, choose a maximal subword $\tilde{y} \prec \sigma^{\tilde{m}-k^2-1}(\xi_1)$ with $\sigma^{k^2+1}(\tilde{y}) \prec \tilde{w}$. 
	Let $x$ (resp. $\tilde{x}$) consist of the first $q := \lceil n/\ell^{k^2+1}-2\rceil$ letters of $y$ (resp. $\tilde{y}$). Then
	\begin{equation}\label{inequality for q}
		\frac{n}{\ell^{k^2+1}} - 2 \ \leq \ q \ \leq \ \frac{n}{\ell^{k^2+1}}.
	\end{equation}  
	
	Write $\sigma(x) = b_1\cdots b_{q\ell}$ and $\sigma(\tilde{x}) = \tilde{b}_1 \cdots \tilde{b}_{q\ell}$, where $b_i, \tilde{b}_i \in \mathcal{A}$.  
	
	Since $M_\sigma$ has a nonzero column, there exists $b \in \mathcal{A}$ such that $b \prec \sigma(a)$ for each $a \in \mathcal{A}$. Thus, for each $i < q$, the word $b_{i\ell+1}\cdots b_{i\ell+\ell}$ must contain at least one $b$, and similarly for the word $\tilde{b}_{i\ell+1}\cdots \tilde{b}_{i\ell+\ell}$. For each $i \in \{0,\ldots, q-1\}$, let $t_i,\tilde{t}_i \in \{1,\ldots, \ell\}$ be the smallest integers such that
	\begin{equation}\label{eqn before label for I_d}
		b_{i\ell + t_i} \ = \ b \ = \ \tilde{b}_{i\ell + \tilde{t}_i}.
	\end{equation}
	Let $d \in \{-\ell + 1,\ldots, \ell - 1\}$ maximize the cardinality of
	\begin{equation}\label{label for I_d}
		I_d \ := \ \bigl\{ i \in \{0,\ldots, q-1\} : \tilde{t}_i-t_i = d \bigr\}.
	\end{equation}
	Then $|I_d| \geq \frac{q}{2\ell-1}$. Without loss of generality, we may assume $d \geq 0$.
	
	Let $Y_1,Y_3, \tilde{Y}_1,\tilde{Y}_3$ be the (possibly empty) words such that
	\begin{equation}
		w \ = \ Y_1\sigma^{k^2+1}(x) Y_3
	\end{equation}
	and
	\begin{equation}
		\tilde{w} \ = \ \tilde{Y}_1\sigma^{k^2+1}(\tilde{x}) \tilde{Y}_3.
	\end{equation}
	Then $\len(Y_1) + \len(Y_3) \leq 2\ell^{k^2+1}$ and $\len(\tilde{Y}_1) + \len(\tilde{Y}_3) \leq 2\ell^{k^2+1}$, so
	\begin{equation} \label{ignoring Y_1 and Y_3}
		\editfull(Y_1,\tilde{Y}_1) + \editfull(Y_3,\tilde{Y}_3) \ \leq \ 2 \ell^{k^2+1}.
	\end{equation}
	
	Let $B_i := \sigma^{k^2}(b_i)$ and $\tilde{B}_i := \sigma^{k^2}(\tilde{b}_i)$. Define $Y_2 := B_{q\ell - d+1}\cdots B_{q\ell}$ and $\tilde{Y}_2 := \tilde{B}_1\cdots \tilde{B}_{d}$, which are both empty if $d = 0$. Then
	\begin{align}
		w \ & = \ Y_1B_1\ldots B_{q\ell-d} Y_2 Y_3, \nonumber \\
		\tilde{w} \ & = \ \tilde{Y}_1\tilde{Y}_2 \tilde{B}_{d+1}\ldots \tilde{B}_{q\ell} \tilde{Y}_3. \label{decomposition of w and tilde w, uniform case}
	\end{align}
	See Figure~\ref{fig: UniformDiagram}.
	\begin{figure}
		\includegraphics[width=\textwidth]{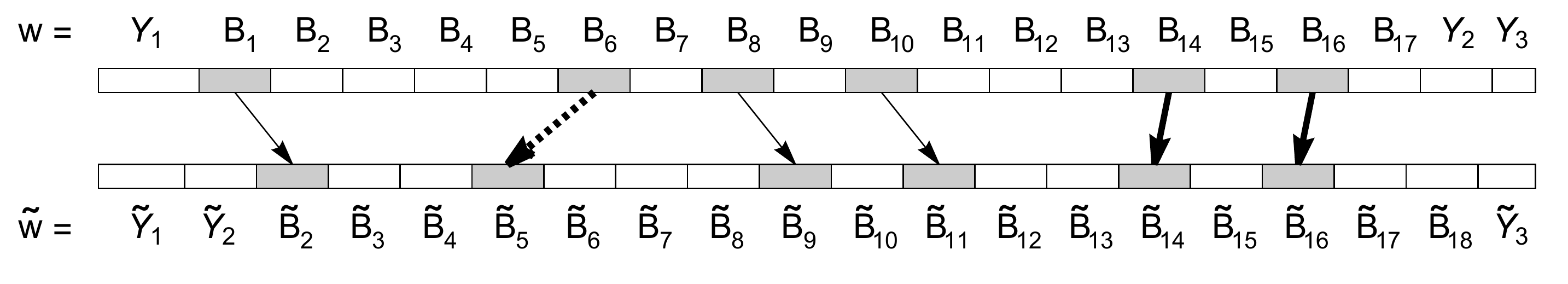}
		\caption{We illustrate \eqref{decomposition of w and tilde w, uniform case} for $q = 6$ and $\ell = 3$. Each triple $B_1B_2B_3$, $B_4B_5B_6$, and so on, contains a copy of the word $\sigma^{k^2}(b)$, indicated in gray. (Other blocks may also be this word.) The most popular shift is $d = 1$, indicated by thin solid arrows. In this case, $I_1 = \{0,2,3\}$.}
		\label{fig: UniformDiagram}
	\end{figure} 
	Moreover,
	\begin{equation}\label{deleting Y_2}
		\editfull(\varnothing,\tilde{Y}_2) + \editfull(Y_2,\varnothing) \ \leq \ \ell^{k^2+1}.
	\end{equation}
	By \eqref{ignoring Y_1 and Y_3}, \eqref{decomposition of w and tilde w, uniform case}, and \eqref{deleting Y_2}, 
	\begin{equation}
		\editfull(w,\tilde{w}) \ \leq \ 3\ell^{k^2+1} + \sum_{j=1}^{q\ell-d} \editfull(B_j,\tilde{B}_{d+j}).
	\end{equation}
	
	Here is the key step: For each $j \in \{1,\ldots, q\ell - d\}$, we have
	\begin{equation}
		\editfull(B_j,\tilde{B}_{d+j}) \ \begin{cases} = 0 & \text{ if } j = i \ell + t_i \text{ for some } i \in I_d, \\
			\leq \ell^{k^2} r_k & \text{ otherwise.} \end{cases}  
	\end{equation}
	There are $q\ell - d - |I_d| \leq q(\ell - \tfrac{1}{2\ell})$ indices $j$ of the second form. Thus, by \eqref{inequality for q},
	\begin{equation}
		\sum_{j=1}^{q\ell-d} \editfull(B_j,\tilde{B}_{d+j}) \ \leq \ \Bigl(1 - \frac{1}{2\ell^2}\Bigr) n r_k
	\end{equation}
	whence
	\begin{equation}
		\frac{\editfull(w,\tilde{w})}{n} \ \leq \ \frac{1}{\ell^k} + \Bigl(1 - \frac{1}{2\ell^2}\Bigr) r_k,
	\end{equation}
	yielding \eqref{eqn in lem: recursive estimate on r_k}.
\end{proof}

\begin{lemma}\label{lem: uniform sub positive incidence matrix bound on rk}
	Let $\ell > 1$ and fix an $\ell$-uniform substitution $\sigma$ with incidence matrix having a nonzero column. Then, for all $k \geq 3$,
	\begin{equation}
		r_k(\sigma) \ \leq \  3 \lambda^k,
	\end{equation} where $\lambda := 1 - \frac{1}{2\ell^2}$.
\end{lemma}
\begin{proof}
	By Lemma~\ref{lem: recursive estimate on r_k}, for each $k \geq 2$,
	\begin{equation}
		r_{k+1} \ \leq \ \frac{1}{\ell^{k}} + \lambda r_k.
	\end{equation}
	It follows that
	\begin{equation}\label{random intermediate equation}
		r_{k+1} \ \leq \ \sum_{i=0}^{k-2} \frac{\lambda^i}{\ell^{(k-i)}} + \lambda^{k-1}r_2.
	\end{equation}
	The identity
	\begin{equation}
		\sum_{i=0}^{k-2} \frac{\lambda^i}{\ell^{(k-i)}} \ = \ \lambda^k \sum_{j=2}^{k} (\ell\lambda)^{-j},
	\end{equation}
	together with \eqref{random intermediate equation} and $r_2 \leq 1$, implies
	\begin{equation}
		r_{k+1} \ \leq \ \lambda^{k+1} \Bigl( \sum_{j=2}^{\infty} (\ell\lambda)^{-j} + \frac{1}{\lambda}\Bigr)/\lambda.
	\end{equation}
	On the domain $[2,\infty)$, the function
	\begin{equation}
		\ell \ \mapsto \ \Bigl( \sum_{j=2}^{\infty} (\ell\lambda)^{-j} + \frac{1}{\lambda}\Bigr)/\lambda
	\end{equation}
	attains its maximum $320/147 \leq 3$ at $\ell = 2$.
\end{proof}

\begin{proof}[Proof of Theorem~\ref{main theorem} in the uniform case]
	Since $|\mathcal{A}| > 1$, the substitution $\sigma$ must be $\ell$-uniform for some $\ell > 1$.
	
	\medskip 
	
	\noindent{\bf Claim}: Let $p \in \N$ be such that $\sigma^p$ has positive incidence matrix. Then, 
	\begin{equation}
		\forall k \geq 4p, \quad r_k(\sigma) \ \leq \ 4 \lambda_p^k,
		\quad \text{where} \; \;  \lambda_p := \left(1-\frac{1}{2\ell^{2p}}\right)^{1/\sqrt{p}}.
	\end{equation}

	To verify the claim, we apply  Lemma~\ref{lem: uniform sub positive incidence matrix bound on rk} to the $\ell^{p}$-uniform substitution $\sigma^{p}$ and find that
	\begin{multline}\label{application of lem: uniform sub positive incidence matrix bound on rk}
		\max\left\{ \frac{\editfull(w,\tilde{w})}{\len(w)} : w,\tilde{w} \in \mathcal{W}_n(\sigma^{p}) \text{ for some } n \in [\ell^{pk^2},\ell^{p(k+1)^2}) \right\} \ \leq \ 3\lambda_*^k
	\end{multline}
	for all $k \geq 3$, where $\lambda_* := 1-\frac{1}{2\ell^{2p}}$.
	
	Let $w,\tilde{w} \in \mathcal{W}_n(\sigma)$ be such that $n \in [\ell^{k^2},\ell^{(k+1)^2})$. Let $m,\tilde{m} \in \N$ and $a_0,\tilde{a}_0 \in \mathcal{A}$ be such that $w \prec \sigma^m(a_0)$ and $\tilde{w} \prec \sigma^{\tilde{m}}(\tilde{a}_0)$.
	
	Let $q \in \N$ be the smallest integer such that $pq - m\geq p$. The matrix $M_\sigma^p$ is positive, hence $M_{\sigma}^{pq-m}$ is positive, so $a_0 \prec \sigma^{pq-m}(a_0)$. Thus we have $w \prec \sigma^m(\sigma^{pq-m}(a_0)) = \sigma^{pq}(a_0) $, so $w \in \mathcal{W}_n(\sigma^p)$, and similarly for $\tilde{w}$.
	
	Observe that $k = \lfloor \sqrt{\log_\ell(n)} \rfloor$. Let $k_1 := \lfloor \sqrt{\log_\ell(n)/p} \rfloor$. Then we have $n \in [\ell^{pk_1^2},\ell^{p(k_1+1)^2})$ and $k_1 \geq k/\sqrt{p} - 1$. Notice that $k \geq 4p$ implies that $k_1 \geq 3$. Therefore, provided that $k \geq 4p$, by \eqref{application of lem: uniform sub positive incidence matrix bound on rk},
	\begin{equation}
		\frac{\editfull(w,\tilde{w})}{\len(w)} \ \leq \ 3\lambda_*^{k_1} \ \leq \ \frac{3}{\lambda_*}  
		\left(\lambda_*^{1/\sqrt{p}}\right)^{k}.
	\end{equation}
	Note that $\lambda_p = \lambda_*^{1/\sqrt{p}}$ and $1/\lambda_* \leq 8/7$. We conclude that, for all $k \geq 4p$,
	\begin{equation}
		\max\left\{ \frac{\editfull(w,\tilde{w})}{\len(w)} : w,\tilde{w} \in \mathcal{W}_n(\sigma) \text{ for some } n \in [\ell^{k^2},\ell^{(k+1)^2}) \right\} \ \leq \ 4\lambda_p^k,
	\end{equation}
	which yields the claim.
\end{proof}

\section{Proof of Theorem~\ref{main theorem} in the non-uniform case}\label{sec: proof of main theorem in non-uniform case}

Let $\sigma$ be a primitive substitution. For integer $k$ and real $\lambda$ with $k,\lambda \geq 2$, define
\begin{equation}\label{definition of r_k, nonuniform case}
	r_k(\sigma,\lambda) := \max\!\left\{ \frac{\editfull(w,\tilde{w})}{\len(w)} : w,\tilde{w} \in \mathcal{W}_n(\sigma) \text{ for some } n \in [\lambda^{k^2},\lambda^{(k+1)^2}) \right\}.
\end{equation}
Since $\lambda \ge 2$, every word of length at least $\lambda^{(k+1)^2}$  can be written as a concatenation of words of length in $[\lambda^{k^2},\lambda^{(k+1)^2})$. Therefore, for all $w,\tilde{w} \in \mathcal{W}_n(\sigma)$ with $n \geq \lambda^{k^2}$, we have
\begin{equation}\label{edit distance and r_k inequality} 
	\editfull(w,\tilde{w}) \ \leq \ r_k(\sigma,\lambda)\len(w).
\end{equation}

\noindent\textbf{Notation.}   Given a non-uniform primitive substitution $\sigma$, let $p=p(\sigma) \in \N$ be minimal such that all entries of the matrix $M_{\sigma}^p$ are positive and write $\tau := \sigma^{p}$. Denote by $\lambda$ the Perron eigenvalue of the positive matrix $M_\tau = M_{\sigma}^p$.

\begin{lemma}\label{lem: non-uniform, word length controlled by PF eval}
	With notation as above, we have $\lambda \geq 2$ and there exist constants $c_1 > 0$ and $C_2 \geq 1$ such that
	\begin{equation}
		c_1\lambda^n \len(w) \ \leq \ \len(\tau^n(w)) \ \leq \ C_2\lambda^n\len(w)
	\end{equation}
	for all $w \in \mathcal{A}^+$ and $n \in \N$.
\end{lemma}
\begin{proof}
	By the Perron--Frobenius theorem (see \cite{hornjohnson}), $M_{\tau}^n/\lambda^n$ converges to a positive matrix as $n \to \infty$. Therefore, there exist $\tilde{c}_1,\tilde{C}_2 > 0$ such that, for all $n$, we have $M_{\tau}^n(a,b) \in [\tilde{c}_1\lambda^n,\tilde{C}_2\lambda^n]$. 
	
	Set $c_1 = \tilde{c}_1|\mathcal{A}|$ and $C_2 = \max\{1,\tilde{C}_2|\mathcal{A}|\}$. Then, for each $w \in \mathcal{A}^+$ and $n \in \N$,
	\begin{equation}
		c_1\lambda^n\len(w) \ \leq \ \len(\tau^{n}(w)) \ \leq \ C_2\lambda^n\len(w).
	\end{equation}
	Moreover, because $\sigma$ is non-uniform, the alphabet $\mathcal{A}$ has at least two symbols, hence, for each $a \in \mathcal{A}$ and $n \in \N$, we observe by positivity of $M_\tau$ that
	\begin{equation}
		2^n \ \leq \ \len(\tau^n(a)) \ \leq \ C_2\lambda^n,
	\end{equation}
	hence $\lambda \geq 2$.
\end{proof}

For a finite set of integers $S =\{s_1,\ldots, s_{|S|}\}$ written in increasing order, we denote the \emph{smallest gap} of $S$ by
\begin{equation}
	\mingap(S) \ := \ \min\{s_{j+1}-s_{j} : 1 \leq j \leq |S| - 1\}.
\end{equation}
and the \emph{largest gap} of $S$ by
\begin{equation}
	\maxgap(S) \ := \ \max\{ s_{j+1} - s_j : 1 \leq j \leq |S|- 1 \}.
\end{equation}

\begin{lemma}\label{lem: recursive estimate on r_k, nonuniform case}
	In the notation preceding Lemma~\ref{lem: non-uniform, word length controlled by PF eval}, there exist constants $C_{10}$ and $k_0$ such that, for each $k \geq k_0$,
	\begin{equation}
		r_{k+1}(\tau,\lambda) \ \leq \ C_{10}/k^2 + \left(1-\frac{2}{k}\right)r_k(\tau,\lambda).
	\end{equation}
\end{lemma}
\begin{remark}
	In the statement of this lemma and in its proof, all constants depend only on $\sigma$.
\end{remark}
\begin{proof}
	Let $c_1,C_2$ be as in Lemma~\ref{lem: non-uniform, word length controlled by PF eval}. 
	
	Fix $k \geq k_0$, with $k_0$ to be determined later. Let $n \in \N$ be such that
	\begin{equation} \label{inequality for n} \lambda^{(k+1)^2} \ \leq \ n  \ <  \ \lambda^{(k+2)^2},
	\end{equation}
	and fix words $w,\tilde{w} \in \mathcal{W}_n(\tau)$. Let $m \in \N$ and $a_0 \in \mathcal{A}$ be such that $w \prec \tau^{m}(a_0)$. We may assume that $m \geq k^2+1$. Indeed, since $w \prec \tau^{m}(a_0)$, for every $i \in \N$, it follows that $w \prec \tau^{m+i}(a_0) = \tau^{m}(\tau^i(a_0))$, since $\tau^i(a_0)$ contains the letter $a_0$ by positivity of $M_\tau$.
	
	Let $x$ be a maximal length subword of $\tau^{m-k^2-1}(a_0)$ such that $\tau^{k^2+1}(x) \prec w$. By maximality, there exists $c_3$ such that
	\begin{equation}
		\len(x) \ \geq \ c_3n/\lambda^{k^2}.
	\end{equation}
	
	Write $\tau(x) = b_1\cdots b_{q}$, where $b_i \in \mathcal{A}$.
	
	For each $i \in \{1,\ldots, q\}$, put $B_i := \tau^{k^2}(b_i)$. Then
	\begin{equation}
		\tau^{k^2+1}(x) \ = \ B_1\cdots B_q.
	\end{equation}
	For all $i \in \{1,\ldots, q\}$,
	\begin{equation}\label{bound on length of B_i}
		\len(B_i) \in [c_1\lambda^{k^2},C_2\lambda^{k^2}].
	\end{equation}
	
	Let $B_0$ and $B_{q+1}$ be the (possibly empty) words such that
	\begin{equation}\label{decomposition of w with B_i}
		w \ = \ B_0B_1\cdots B_qB_{q+1}.
	\end{equation}
	By the maximality of $x$,
	\begin{equation}\label{eqn: bound on B_i ends}
		\max\{\len(B_0),\len(B_{q+1})\} \ \leq \ C_2\lambda^{k^2+1}.
	\end{equation}
	
	Fix $b \in \mathcal{A}$, and suppose $a$ denotes the first symbol of $\tau^{k^2}(b)$. Let
	\begin{equation}
		I_1 \ := \ \{\nu \in \{1,\ldots, q\} : b_\nu = b \}.
	\end{equation}
	By the positivity of $M_\tau$, we see that $|I_1| \geq \len(x)$ and hence $|I_1| \geq c_3n/\lambda^{k^2}$. Moreover, writing $\ell := \max\{\tau(\xi) : \xi \in \mathcal{A}\}$, we have
	\begin{equation}\label{mingap for I_1}
		\mingap(I_1) \ \geq \ 1
	\end{equation}
	and
	\begin{equation}\label{maxgap for I_1}
		\maxgap(I_1) \ \leq \ 2\ell.
	\end{equation}
	
	For each $\nu \in I_1$, let us call the first $a$ in $B_\nu$ the ``initial $a$''; the index in $w$ of the initial $a$ of $B_\nu$ is
	\begin{equation}
		1 +\sum_{j=0}^{\nu-1} \len(B_j).
	\end{equation}
	Let
	\begin{equation}
		I_2 \ := \ \Bigl\{1 +\sum_{j=0}^{\nu-1} \len(B_j) : \nu \in I_1 \Bigr\}
	\end{equation}
	and note that $|I_2| = |I_1|$.
	By \eqref{bound on length of B_i}, we have by \eqref{mingap for I_1}
	\begin{equation}\label{mingap for I_2}
		\mingap(I_2) \ \geq \ c_1\lambda^{k^2}\mingap(I_1) \ \geq \ c_1\lambda^{k^2}
	\end{equation}
	and by \eqref{maxgap for I_1}
	\begin{equation}\label{maxgap for I_2} \maxgap(I_2) \ \leq \ C_2\lambda^{k^2} \maxgap(I_1) \ \leq \ 2\ell C_2\lambda^{k^2} .
	\end{equation}
	
	\noindent Analogously, there is a word $\tilde{x}$ such that $\tau^{k^2+1}(\tilde{x}) \prec \tilde{w}$ and $\len(\tilde{x}) \geq c_3n/\lambda^{k^2}$. Writing $\tau(\tilde{x}) = \tilde{b}_1\cdots \tilde{b}_{\tilde{q}}$, where $\tilde{b}_i \in \mathcal{A}$, and $\tilde{B}_i := \tau^{k^2}(\tilde{b}_i)$, we have
	\begin{equation}
		\tau^{k^2+1}(\tilde{x}) \ = \ \tilde{B}_1\cdots \tilde{B}_{\tilde{q}}
	\end{equation}
	and, for all $i \in \{1,\ldots, \tilde{q}\}$,
	\begin{equation}\label{bound on length of tilde B_i}
		\len(\tilde{B}_i) \in [c_1\lambda^{k^2},C_2\lambda^{k^2}].
	\end{equation}
	Analogously, there exist (possibly empty) words $\tilde{B}_0$ and $\tilde{B}_{\tilde{q}+1}$ such that
	\begin{equation}\label{concatenation of tilde w}
		\tilde{w} \ = \ \tilde{B}_0\tilde{B}_1\cdots \tilde{B}_{\tilde{q}}\tilde{B}_{\tilde{q}+1}
	\end{equation}
	and
	\begin{equation} \label{eqn: bound on tilde B_i ends}
		\max\{\len(\tilde{B}_0),\len(\tilde{B}_{\tilde{q}+1})\} \ \leq \ C_2\lambda^{k^2+1}.
	\end{equation}
	
	Moreover, defining
	\begin{equation}
		\tilde{I}_1 \ := \ \{\nu \in \{1,\ldots, \tilde{q}\} : \tilde{b}_\nu = b \},
	\end{equation}
	we again have $|I_1| \geq c_3n/\lambda^{k^2}$, and defining
	\begin{equation}
		\tilde{I}_2 \ := \ \Bigl\{1 +\sum_{j=0}^{\nu-1} \len(\tilde{B}_j) : \nu \in \tilde{I}_1 \Bigr\},
	\end{equation}
	we likewise have $|\tilde{I}_2| = |\tilde{I}_1|$,
	\begin{equation}
		\mingap(\tilde{I}_2) \ \geq \ c_1\lambda^{k^2},
	\end{equation}
	and
	\begin{equation}\label{spacing of tilde I_2 elements}
		\maxgap(\tilde{I}_2) \ \leq \ 2\ell C_2\lambda^{k^2}.
	\end{equation}
	Since $\min(\tilde{I}_1) \leq \ell$, we also observe that
	\begin{equation}
		\min(\tilde{I}_2) \ \leq \ 1 + \sum_{j=0}^{\ell-1} \len(\tilde{B}_j),
	\end{equation}
	hence, by \eqref{bound on length of tilde B_i} and \eqref{eqn: bound on tilde B_i ends} and the fact that $C_2\lambda^{k^2} \geq 1$, 
	\begin{equation}\label{min of tilde I_2}
		\min(\tilde{I}_2) \ \leq \ (\lambda+\ell) C_2\lambda^{k^2}.
	\end{equation}
	\begin{figure}
		\includegraphics[width=\textwidth]{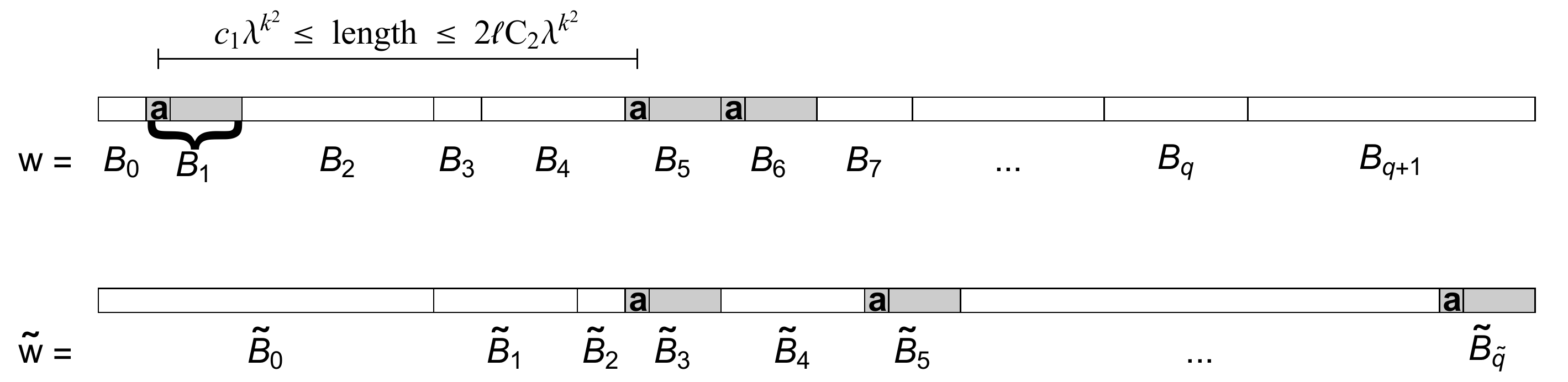}
		\captionsetup{singlelinecheck=off}
		\caption{The word $w$ of length $n$ is composed of words $B_i$ for $0\leq i \leq q+1$; see \eqref{bound on length of B_i}, \eqref{decomposition of w with B_i}, and \eqref{eqn: bound on B_i ends}. The words $B_0$ and $B_{q+1}$ can be empty. We fix $b \in \mathcal{A}$ and consider the set $I_1$ of indices $i$ for which $B_i = \tau^{k^2}(b)$. The subwords $B_i$ with $i \in I_1$ appear in gray and start with the letter $a$. The set $I_2$ consists of the indices in $w$ of these ``initial $a$'s.'' (The letter $a$ also appears in other locations in $w$.) Similar statements hold for $\tilde{w}$.  }
		\label{fig: Picture3}
	\end{figure}
	The situation so far is summarized by Figure~\ref{fig: Picture3}.
	
	Before we continue, let us describe the general strategy of the proof of this lemma. We express both $w$ and $\tilde{w}$ as concatenations of the same large number of copies of the word $\tau^{k^2}(b)$, separated by words of length at least $\lambda^{k^2}$. Then $\editfull(w,\tilde{w})$ is at most the sum of the edit distances between corresponding words in these decompositions. On the one hand, each copy of $\tau^{k^2}(b)$ in the decomposition for $w$ is matched up to a corresponding copy in the decomposition for $\tilde{w}$, and such word pairs contribute zero edit distance. On the other hand, we estimate the edit distances contributed by the remaining word pairs using (essentially only) $r_k(\tau,\lambda)$.
	
	The sets $I_2$ and $\tilde{I}_2$ are not suitable to help define the nice decompositions of $w$ and $\tilde{w}$ described in the previous paragraph, so we modify them now.
	
	Without loss of generality, we may assume that $\max(I_2) \leq \max(\tilde{I}_2)$.
	
	Write $L := \len(\tau^{k^2}(b))$ and $\delta:=  L/\lambda^{k^2}$, so $\delta \in [c_1,C_2]$ by   \eqref{bound on length of B_i}.  Define $I_3$ by taking every $\lceil (2+\delta)/c_1 \rceil$-th element of $I_2$. Then there exists $c_4$ such that $|I_3| \geq c_4|I_2|$. Moreover, by \eqref{mingap for I_2}
	\begin{equation} \label{mingap for I_3}
		\mingap(I_3) \ \geq \ \lceil (2+\delta)/c_1 \rceil\mingap(I_2) \ \geq \ 2\lambda^{k^2}+L.
	\end{equation}
	
	Let $C_5 := \max\{2\ell C_2, (\lambda+\ell)C_2\}$.
	
	\noindent{\bf Claim:} The map $\psi:I_3 \to \Z$ given by
	\begin{equation} 
		\psi(i) \ := \ \min\{ j \ge 0: \, i + j \in \tilde{I}_2\}
	\end{equation}
	satisfies
	\begin{equation} \label{bound on psi}
		\psi(i) \ \leq \ C_5\lambda^{k^2}
	\end{equation} for all $i \in I_3$.
	
	Note $\psi$ is well defined because $\max(I_3) \leq \max(I_2) \leq \max(\tilde{I}_2)$. 
	
	Observe that for all $i \in I_3$ with $i < \min(\tilde{I}_2)$, we have
	\begin{equation} \psi(i) \ = \ \min(\tilde{I}_2) - i \ \leq \ \min(\tilde{I}_2),
	\end{equation}
	while for all $i \in I_3$ with $i \geq \min(\tilde{I}_2)$, we have
	\begin{equation} \psi(i) \ \leq \ \maxgap(\tilde{I}_2)
	\end{equation} since $i \leq \max(\tilde{I}_2)$. The claim thus follows by \eqref{spacing of tilde I_2 elements} and \eqref{min of tilde I_2}. 
	
	Define $\tilde{I}_3 := \{i + \psi(i) : i \in I_3\}$. 
	
	We now present nice decompositions of $w$ and $\tilde{w}$. Write $s := |I_3| = |\tilde{I}_3|$. We decompose $w$ and $\tilde{w}$ as
	\begin{align}\label{initial decomposition of w and tilde w}
		w \ & = \ w_0x_1w_1 \cdots x_sw_s \nonumber \\
		\tilde{w} \ & = \ \tilde{w}_0\tilde{x}_1\tilde{w}_1 \cdots \tilde{x}_s\tilde{w}_s,
	\end{align}
	where
	\begin{enumerate}
		\item each of the words $x_1, \ldots, x_s, \tilde{x}_1, \ldots, \tilde{x}_s$ is the word $\tau^{k^2}(b)$,
		\item for each $j$, the index in $w$ of the initial $a$ of $x_j$ is an element of $I_3$,
		\item for each $j$, the index in $\tilde{w}$ of the initial $a$ of $\tilde{x}_j$ is an element of $\tilde{I}_3$,
		\item $\psi(i) \in [0,C_5\lambda^{k^2}]$ for all $i \in I_3$,
		\item and each of the words $w_1,\ldots, w_{s-1}$ has length at least $\lambda^{k^2}$ by \eqref{mingap for I_3}.
	\end{enumerate}

	Since $|I_3| \geq c_4|I_2|$ and $|I_2| = |I_1| \geq c_3n/\lambda^{k^2}$, there exists $c_7$ such that $|I_3| \geq c_7 n/\lambda^{k^2}$.
	
	Let $\theta>0$ be a constant to be determined later.
	
	Let $C_8 = C_5 + 1$. Observe that
	\begin{equation}\label{bins}
		[0,C_5] \ \subset \ \bigcup_{j=1}^{\lceil C_8k/\theta\rceil} \left[\frac{(j-1)\theta}{k},\frac{j\theta}{k}\right).
	\end{equation}
	For each $j$, define
	\begin{equation}\label{bins for I_3}
		I_3(j) \ := \ \left\{ i \in I_3 : \frac{\psi(i)}{\lambda^{k^2}} \in \left[\frac{(j-1)\theta}{k},\frac{j\theta}{k}\right) \right\}.
	\end{equation}
	Fix $j_*$ such that $|I_3(j_*)|$ is maximized. Then, by the pigeonhole principle,
	\begin{equation}
		|I_3(j_*)| \ \geq \ \frac{|I_3|}{\lceil C_8k/\theta\rceil},
	\end{equation}
	hence there exists $c_9$ such that
	\begin{equation}
		|I_3(j_*)| \ \geq \ c_9\frac{\theta n}{k\lambda^{k^2}}.
	\end{equation}
	
	Fix a subset $\hat{I}_3$ of $I_3(j_*)$ with cardinality
	\begin{equation}\label{cardinality of s_*}
		s_* \ := \ \left\lceil c_9\frac{\theta n}{k\lambda^{k^2}} \right\rceil.
	\end{equation}
	Enumerate $\hat{I}_3$ in increasing order as $\iota_1,\ldots, \iota_{s_*}$, and let $i_j$ be such that the index of the initial $a$ of $x_{i_j}$ is precisely $\iota_{j}$.
	
	Now let us prepare to estimate $\editfull(w,\tilde{w})$.
	
	First, we modify the decomposition \eqref{initial decomposition of w and tilde w} by combining some words together. Define the words $W_0,W_1,\ldots, W_{s_*}$ as follows:
	\begin{equation} 
		w \ = \ \underbrace{w_0x_1w_1 \cdots x_{i_1-1}w_{i_1-1}}_{W_0} x_{i_1} \underbrace{w_{i_1}\cdots w_{i_2-1}}_{W_1} x_{i_2} \cdots x_{i_{s_*}} \underbrace{w_{i_{s_*}}\cdots w_s}_{W_{s_*}}.
	\end{equation}
	Do the same for $\tilde{w}$, so that (see Figure~\ref{fig: UpdatedPicture10})
	\begin{align}\label{final decomposition of w and tilde w}
		w \ & = \ W_0x_{i_1}W_1x_{i_2}W_2\cdots x_{i_{s_*}}W_{s_*}, \nonumber \\
		\tilde{w} \ & = \ \tilde{W}_0\tilde{x}_{i_1}\tilde{W}_1\tilde{x}_{i_2}\tilde{W}_2\cdots \tilde{x}_{i_{s_*}}\tilde{W}_{s_*}.
	\end{align}
	\begin{figure}
		\includegraphics[width=\textwidth]{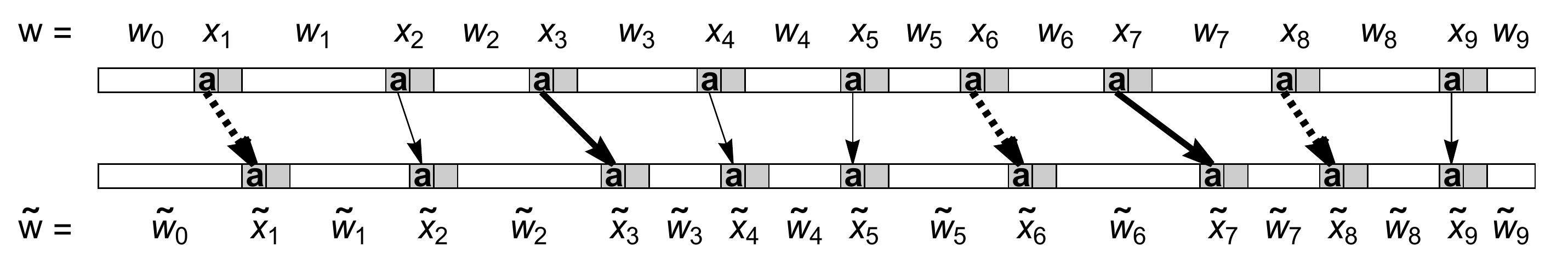}
		\caption{We illustrate \eqref{initial decomposition of w and tilde w} for $s = 9$.  The set $I_3$ consists of the locations in $w$ of the initial $a$'s in the words $x_1,\ldots, x_9$, and similarly for $\tilde{I}_3$. For each $i \in I_3$,  the horizontal length of the corresponding arrow is   $\psi(i)$. All ratios $\psi(i)/\lambda^{k^2}$ belong to $[0,C_5]$, which we divide into bins $I_3(j)$ (see \eqref{bins} and \eqref{bins for I_3}). In this figure, the ratios fall into three bins, and the corresponding arrows are marked accordingly. The most crowded bin $I_3(j_*)$ is represented by the thin solid arrows.}
		\label{fig: Picture9}
	\end{figure} 
	\begin{figure}
		\includegraphics[width=\textwidth]{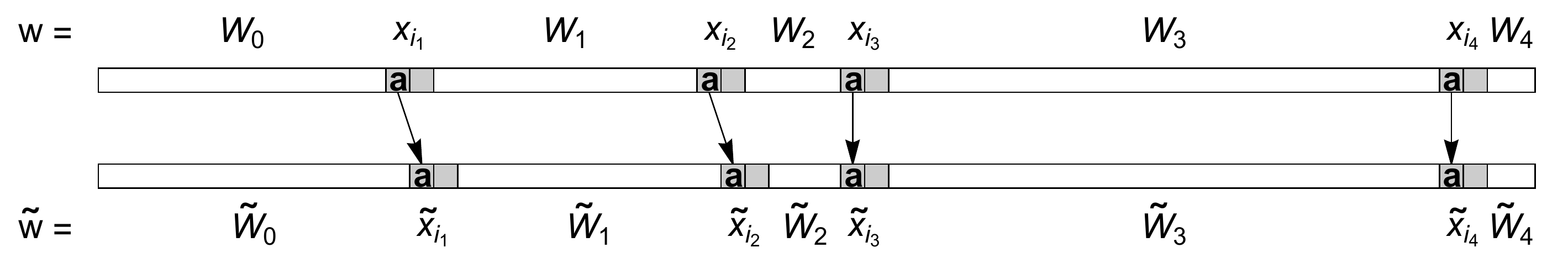}
		\caption{To obtain the decomposition \eqref{final decomposition of w and tilde w}, we ignore all $x_i$ except those with ratio $\psi(i)/\lambda^{k^2} \in I_3(j_*)$ and merge words from the previous decomposition \eqref{initial decomposition of w and tilde w} as appropriate. For example, $W_0$ in this figure is $w_0x_1w_1$ in Figure~\ref{fig: Picture9}. (To simplify this caption, we assume $|\hat{I}_3| = |I_3(j_*)| = s_* = 4$.)}
		\label{fig: UpdatedPicture10}
	\end{figure}

	Now, for each $j \in \{0,\ldots, s_*\}$, define $y_j,\tilde{y}_j,z_j,\tilde{z}_j$ to be the words such that the conditions
	\begin{align}
		W_{j} \ & = \ y_jz_j \nonumber \\
		\tilde{W}_{j} \ & = \ \tilde{y}_j\tilde{z}_j \label{decomposition of W_j} 
	\end{align}
	and
	\begin{equation}
		\len(y_j) \ = \ \len(\tilde{y}_j) \ = \ \min\{\len(W_j),\len(\tilde{W}_j)\} 
	\end{equation}
	hold. At least one of $z_j$ and $\tilde{z}_j$ will be the empty word. We observe that, for $j \neq 0,s_*$, 
	\begin{equation}
		\max\{\len(z_j),\len(\tilde{z}_j)\} \ = \ \left| \len(W_j) - \len(\tilde{W}_j) \right| \ \leq \ \frac{\theta \lambda^{k^2}}{k}.
	\end{equation}
	Since the edit distance between two words is trivially bounded by (half) the sum of their lengths, we have for such $j$
	\begin{equation}\label{edit distance for z_j}
		\editfull(z_j,\tilde{z}_j) \ \leq \ \frac{\theta\lambda^{k^2}}{k}.
	\end{equation}
	
	Next, observe that since the words $w_1,\ldots, w_{s-1}$ have length at least $\lambda^{k^2}$, it follows that for all $j \in \{1,\ldots, s_*-1\}$,
	\begin{equation}
		\len(W_j) \ \geq \ \lambda^{k^2}.
	\end{equation}
	
	Set $\theta = 2/c_1c_9$. Then there exists $k_1$ such that $k \geq k_1$ implies 
	\begin{equation}
		\frac{\theta}{k} \ \leq \ 1.
	\end{equation}
	Fix $j \in \{1,\ldots, s_*-1\}$ and assume $k \geq k_1$. First,
	\begin{equation}
		\len(\tilde{W}_j) \ = \ \bigl( \iota_{j+1} + \psi(\iota_{j+1}) \bigr) - \bigl( \iota_{j} + \psi(\iota_{j}) + L\bigr).
	\end{equation}
	Second, by \eqref{mingap for I_3},
	\begin{equation}
		\iota_{j+1} - \iota_{j} \ \geq \ 2\lambda^{k^2}+L.
	\end{equation}
	Third, since $\iota_{j},\iota_{j+1} \in I_3(j_*)$,
	\begin{equation}
		\left\vert \psi(\iota_{j+1}) - \psi(\iota_{j}) \right\vert \ \leq \ \frac{\theta\lambda^{k^2}}{k} \ \leq \ \lambda^{k^2}.
	\end{equation}
	Therefore, $\len(\tilde{W}_j) \geq \lambda^{k^2}$.
	
	Therefore, provided $k \geq k_1$, for $j\in\{1,\ldots,s_*-1\}$,
	\begin{equation}
		\len(y_j) \ \geq \ \lambda^{k^2}, 
	\end{equation}
	so, by \eqref{edit distance and r_k inequality},
	\begin{equation}\label{edit distance for y_j}
		\editfull(y_j,\tilde{y}_j) \ \leq \ r_k(\tau,\lambda)\len(y_j).
	\end{equation}
	
	Next, we observe that
	\begin{equation}\label{identity for length of W_0}
		\len(\tilde{W}_0) - \len(W_0) \ = \ \psi(\iota_1) \ \leq \ C_5\lambda^{k^2}.
	\end{equation}
	If $\len(W_0) \geq \lambda^{k^2}$, then $\len(y_0) \geq \lambda^{k^2}$ and
	\begin{equation} \max\{\len(z_0),\len(\tilde{z}_0)\} \ = \ \psi(\iota_1) \ \leq \ C_5\lambda^{k^2},
	\end{equation}
	hence
	\begin{equation}
		\editfull(W_0,\tilde{W}_0) \ \leq \ \editfull(y_0,\tilde{y}_0) + \editfull(z_0,\tilde{z}_0) \ \leq \ r_k(\tau,\lambda) \len(y_0) + C_5\lambda^{k^2},
	\end{equation}
	and if $\len(W_0) < \lambda^{k^2}$, then by \eqref{identity for length of W_0}
	\begin{equation}
		\max\{\len(W_0),\len(\tilde{W}_0)\} \ = \ \len(\tilde{W}_0) \ < \ \lambda^{k^2} + C_5\lambda^{k^2},
	\end{equation}
	so that
	\begin{equation}
		\editfull(W_0,\tilde{W}_0) \ < \ \lambda^{k^2} + C_5\lambda^{k^2}.
	\end{equation}
	Thus, unconditionally, and recalling that $C_8 = C_5 + 1$, we have
	\begin{equation}\label{edit distance for left end}
		\editfull(W_0,\tilde{W}_0) \ \leq \ r_k(\tau,\lambda) \len(y_0) + C_8\lambda^{k^2}.
	\end{equation}
	Analogously, starting from the fact that $\len(W_{s_*})-\len(\tilde{W}_{s_*}) = \psi(\iota_{s_*})$, we obtain
	\begin{equation}\label{edit distance for right end}
		\editfull(W_{s_*},\tilde{W}_{s_*}) \ \leq \ r_k(\tau,\lambda) \len(y_{s_*}) + C_8\lambda^{k^2}.
	\end{equation}
	
	\medskip
	
	Now we are ready to estimate $\editfull(w,\tilde{w})$. By construction, all $x_i$ and $\tilde{x}_i$ are the word $\tau^{k^2}(b)$, hence, for each $j$,
	\begin{equation}\label{edit distance for x_j} 
		\editfull(x_{i_j},\tilde{x}_{i_j}) \ = \ 0.	
	\end{equation}
	By the decompositions \eqref{final decomposition of w and tilde w} and \eqref{decomposition of W_j} and the relations \eqref{edit distance for y_j}, \eqref{edit distance for z_j}, \eqref{edit distance for x_j}, \eqref{edit distance for left end}, and \eqref{edit distance for right end}, we obtain (provided $k \geq k_1$)
	\begin{multline}
		\editfull(w,\tilde{w}) \ \leq \ \editfull(W_0,\tilde{W}_0) +
		\editfull(x_{i_1},\tilde{x}_{i_1}) + \editfull(y_1,\tilde{y}_1) + \editfull(z_1,\tilde{z}_1) \\ + \editfull(x_{i_2},\tilde{x}_{i_2}) + \editfull(y_2,\tilde{y}_2) + \editfull(z_2,\tilde{z}_2) \\ + \cdots + \\
		+\editfull(x_{i_{s_*-1}},\tilde{x}_{i_{s_*-1}}) + \editfull(y_{s_*-1},\tilde{y}_{s_*-1}) + \editfull(z_{s_*-1},\tilde{z}_{s_*-1}) \\
		+\editfull(x_{i_{s_*}},\tilde{x}_{i_{s_*}}) +  \editfull(W_{s_*},\tilde{W}_{s_*}) \\
		\leq \ 2C_8\lambda^{k^2} + r_k(\tau,\lambda) \sum_{j=0}^{s_*} \len(y_{j}) + (s_*-1)\frac{\theta \lambda^{k^2}}{k}.
	\end{multline}
	Recall $L = \len(\tau^{k^2}(b))$. Note that 
	\begin{equation}
		\sum_{j=0}^{s_*} \len(y_j) \ = \  n - \Bigl( s_*L + \sum_{j=0}^{s_*} \len(z_j) \Bigr)
		\ \leq \ n - s_*L.
	\end{equation}
	Therefore, we have
	\begin{equation}
		\frac{\editfull(w,\tilde{w})}{n} \ \leq \ \frac{2C_8\lambda^{k^2}}{n} +  \frac{(s_*-1)\theta\lambda^{k^2}}{nk} + r_k(\tau,\lambda) \left( 1 - \frac{s_*L}{n} \right).
	\end{equation}
	Let us estimate each term in the previous inequality.
	
	Since $L \geq c_1\lambda^{k^2}$ and $s_* \geq c_9\frac{\theta n}{k\lambda^{k^2}}$,
	\begin{equation}
		\frac{s_*L}{n} \ \geq \ \frac{c_1c_9\theta}{k}.
	\end{equation}
	Recall $\theta = 2/c_1c_9$. Then
	\begin{equation} \label{needed inequality 2}
		r_k(\tau,\lambda) \left( 1 - \frac{s_*L}{n} \right) \ \leq \ \left(1 - \frac{2}{k}\right)r_k(\tau,\lambda).
	\end{equation}
	Next, recall \eqref{inequality for n} and \eqref{cardinality of s_*}. Observe that there exists $k_2$ and $C_{10}$ such that, for all $k \geq k_2$,
	\begin{equation}\label{needed inequality 1}
		\frac{2C_8\lambda^{k^2}}{n} + 	\frac{(s_*-1)\theta \lambda^{k^2}}{nk} \ \leq \ \frac{C_{10}}{k^2}.
	\end{equation}
	Set $k_0= \max\{k_1,k_2\}$. Thus, by \eqref{needed inequality 2} and \eqref{needed inequality 1}, for all $k \geq k_0$,
	\begin{equation}
		\frac{\editfull(w,\tilde{w})}{n} \ \leq \ \frac{C_{10}}{k^2} + \left( 1 - \frac{2}{k}\right) r_k(\tau,\lambda).
	\end{equation}
	None of the constants depend on the words $w$ and $\tilde{w}$, so the lemma follows.
\end{proof}

\begin{lemma}\label{lem: bound on r_k, non-uniform case}
	In the notation preceding Lemma~\ref{lem: non-uniform, word length controlled by PF eval}, there exist constants $C_{13}$ and $k_0$ such that for each $k \geq k_0$,
	\begin{equation}
		r_{k}(\tau,\lambda) \ \leq \ C_{13}/k.
	\end{equation}
\end{lemma}
\begin{proof}
	By Lemma~\ref{lem: recursive estimate on r_k, nonuniform case}, there are constants $C_{10}$ and $k_0$ such that, for each $k \geq k_0$,
	\begin{equation}\label{eqn: consequence of lem: recursive estimate on r_k, nonuniform case}
		r_{k+1}(\tau,\lambda) \ \leq \ C_{10}/k^2 + \left(1-\frac{2}{k}\right)r_k(\tau,\lambda).
	\end{equation}
	
	For each $k \geq k_0$, define
	\begin{equation}
		z_k \ := \ r_k(\tau,\lambda) \prod_{j=k_0}^{k-1} \frac{j}{j-2} \ = \ r_k(\tau,\lambda) \frac{(k-1)(k-2)}{(k_0-1)(k_0-2)}.
	\end{equation}
	Then $z_{k_0} = r_{k_0}(\tau,\lambda) \leq 1$ trivially, and there exists $\tilde{c}_{11}$ such that, for all $k \geq k_0$,
	\begin{equation} \label{eq:z_k-to-r_k}
		z_k \ > \ \tilde{c}_{11} r_k(\tau,\lambda)k^2.
	\end{equation}
	Multiply each side of \eqref{eqn: consequence of lem: recursive estimate on r_k, nonuniform case} by $\prod_{j=k_0}^{k} \frac{j}{j-2}$ to obtain the inequality
	\begin{equation}
		z_{k+1} \ \leq \ \frac{C_{10}}{k^2} \cdot \frac{k(k-1)}{(k_0-1)(k_0-2)} + z_k \ \leq \ C_{12} + z_k , 
	\end{equation}
	where we may assume that $C_{12} \ge 1$. Therefore,
	\begin{equation}
		z_{k} \ \leq \ (k-k_0)C_{12} + z_{k_0} \ \leq \ (k-k_0)C_{12} + 1 \ \leq \ C_{12}k.
	\end{equation}
	Combining this with \eqref{eq:z_k-to-r_k}, we conclude that there exists $C_{13}$ such that
	\begin{equation}
		r_{k} \ \leq \ C_{13}/k. 
	\end{equation}
\end{proof}

\begin{lemma} \label{lem: words are words} In the notation preceding Lemma~\ref{lem: non-uniform, word length controlled by PF eval}, $r_k(\sigma,\lambda) = r_k(\tau,\lambda)$.
\end{lemma}
\begin{proof}
	The result follows by showing that $\mathcal{W}_n(\sigma^p) = \mathcal{W}_n(\sigma)$ for all $n \in \N$.
	
	The inclusion $\mathcal{W}_n(\sigma^p) \subseteq \mathcal{W}_n(\sigma)$ is trivial.
	
	Fix $w \in \mathcal{W}_n(\sigma)$, and let $m \in \N$ and $a_0 \in \mathcal{A}$ be such that $w \prec \sigma^m(a_0)$. Let $q \in \N$ be the smallest integer such that $pq - m\geq p$. The matrix $M_{\sigma}^{p}$ is positive, hence $M_{\sigma}^{pq-m}$ is positive, so the word $\sigma^{pq-m}(a_0)$ contains the letter $a_0$. Thus $w \prec \sigma^{pq}(a_0) = \sigma^m(\sigma^{pq-m}(a_0))$. Therefore $w \in \mathcal{W}_n(\sigma^{p})$.
\end{proof}

Now we can prove Theorem~\ref{main theorem}.  
\begin{proof}[Proof of Theorem~\ref{main theorem} when $\sigma$ is non-uniform]
	By Lemmas~\ref{lem: bound on r_k, non-uniform case}~and~\ref{lem: words are words}, there exist constants $\lambda \geq 2$ and $C$, $k_0$ such that for all $k \geq k_0$,
	\begin{equation}
		r_k(\sigma,\lambda) \ \leq \ C/k.
	\end{equation}
	For clarity, note that $\lambda$ is the Perron eigenvalue of the matrix $M_{\sigma}^p$, where $p \in \N$ is minimal such that all entries in $M_\sigma^p$ are positive.
	It follows that
	\begin{equation}
		\mathrm{diam}_E(\mathcal{W}_n(\sigma)) \ = \ O\Bigl(\frac{n}{\sqrt{\log n}}\Bigr).
	\end{equation}
\end{proof}

\section{Beyond primitive substitutions}\label{sec: proof of characterization in uniform case}
The goal of this section is to prove Proposition~\ref{introprop: characterization of uniform case}.

Let $\sigma$ be a uniform substitution. Suppose there is a unique essential strongly connected component $G_0$ of $G_\sigma$, and let $\mathcal{A}_0$ denote the set of letters corresponding to the vertices of $G_0$. For each word $w = a_1\cdots a_m$ with $a_i \in \mathcal{A}$, write
\begin{equation}
	\rho(w) \ := \ \frac{\#\{ i \in \{1,\ldots, m\} : a_i \in \mathcal{A}\setminus \mathcal{A}_0\}}{m}.
\end{equation}

First, we need the following two lemmas, which assure us that letters from $\mathcal{A}_0$ predominate in words generated by $\sigma$, a fact which depends crucially on the uniformity of $\sigma$.

\begin{lemma}\label{lem: unif case, ess letters dominate sigma iterates}
	Let $\sigma$ be an $\ell$-uniform substitution such that $\ell > 1$ and there is a unique essential strongly connected component $G_0$ of $G_\sigma$. Then, for each $a \in \mathcal{A}$,
	\begin{equation}
		\lim_{n \to \infty} \rho(\sigma^n(a)) \ = \ 0.
	\end{equation}
\end{lemma}
\begin{proof} 
	By the assumption that $G_0$ exists and is unique, there exists $K \in \N$ such that, for all $a \in \mathcal{A}$, the word $\sigma^K(a)$ contains at least one letter from $\mathcal{A}_0$. Hence
	\begin{equation}\label{uniform sub maximum iness ratio}
		\max \{ \rho(\sigma^K(a)) : a \in \mathcal{A}\setminus \mathcal{A}_0 \} \ \leq \ \frac{\ell^K - 1}{\ell^K}.
	\end{equation}
	
	Let $w \in \mathcal{A}^+$, and let $E_w$ (resp. $I_w$) be the number of letters of $w$ that belong to $\mathcal{A}_0$ (resp. $\mathcal{A} \setminus \mathcal{A}_0$). By \eqref{uniform sub maximum iness ratio} and the fact that, for each $a \in \mathcal{A}_0$, the word $\sigma(a)$ consists entirely of letters from $\mathcal{A}_0$, it follows that
	\begin{multline}\label{uniform sub iness ratio contraction}
		\rho(\sigma^K(w)) \ \leq \ \frac{E_w\ell^K\cdot 0+I_w\ell^K \cdot \max\{ \rho(\sigma^K(a)): a \in \mathcal{A}\setminus \mathcal{A}_0  \} }{\len(w) \ell^K} \\ \leq \ \rho(w)\left(1-\frac{1}{\ell^K}\right).
	\end{multline}
	
	The lemma's conclusion is trivial to verify when $a \in \mathcal{A}_0$, so suppose $a \not\in \mathcal{A}_0$. For each $n \in \N$, write $\rho_n := \rho(\sigma^{n}(a))$. By \eqref{uniform sub maximum iness ratio}, $\rho_{K} \ \leq \ 1-1/\ell^K$. By \eqref{uniform sub iness ratio contraction}, 
	\begin{equation} \rho_{(n+1)K} \ \leq \ \left(1-\frac{1}{\ell^K}\right) \rho_{nK}
	\end{equation}
	for all $n \in \N$. Therefore, 
	\begin{equation}
		\lim_{n \to \infty} \rho(\sigma^{Kn}(a)) \ = \ 0.
	\end{equation}
	Moreover, since, for each $a \in \mathcal{A}_0$, the word $\sigma(a)$ consists entirely of letters from $\mathcal{A}_0$ by assumption on $G_0$, we have $\rho(\sigma(w)) \leq \rho(w)$ for all words $w$, hence
	\begin{equation}
		\lim_{n\to\infty} \rho(\sigma^n(a)) \ = \ 0
	\end{equation}
	as well.
\end{proof}

\begin{lemma}\label{lem: unif case, ess letters dominate general words}
	Let $\sigma$ be an $\ell$-uniform substitution such that $\ell > 1$ and there is a unique essential strongly connected component $G_0$ of $G_\sigma$. Then
	\begin{equation}
		\lim_{n\to\infty} \max_{w \in \mathcal{W}_n(\sigma)} \rho(w) \ = \ 0.
	\end{equation}
\end{lemma}
\begin{proof}
	Fix $w \in \mathcal{W}_n(\sigma)$. Let $m \in \N$ and $a \in \mathcal{A}$ be such that $w \prec \sigma^m(a)$. Let $k \in \N\cup\{0\}$ be such that $\ell^{2k} \leq n < \ell^{2k+2}$; then $m \geq 2k$. Let $x$ be a maximal subword of $\sigma^{m-k}(a)$ such that $\sigma^k(x) \prec w$. Then
	\begin{equation}
		\ell^{k}\len(x) \ \leq \ n \ \leq \ \ell^k(\len(x)+2).
	\end{equation}
	Let
	\begin{equation}
		\gamma_k \ := \ \max_{b \in \mathcal{A}} \rho(\sigma^k(b)).
	\end{equation}
	By Lemma~\ref{lem: unif case, ess letters dominate sigma iterates},
	\begin{equation}\label{gamma_k tends to zero}
		\lim_{k \to \infty} \gamma_k \ = \ 0.
	\end{equation}
	Let $Y_1$ and $Y_3$ be the possibly empty words of length at most $\ell^k$ such that
	\begin{equation}
		w \ = \ Y_1\sigma^k(x)Y_3.
	\end{equation}
	The number of letters in $w$ that do not belong to $\mathcal{A}_0$ is at most
	\begin{equation}
		\ell^k\len(x)\cdot \gamma_k + 2\ell^k \ \leq \ n\left(\gamma_k + \frac{2}{\sqrt n}\right),
	\end{equation}
	hence is $o(n)$ by \eqref{gamma_k tends to zero}. Thus, $\rho(w) = o(1)$ independent of $w$.
\end{proof}

Suppose $\sigma$ is a substitution such that there is a unique essential strongly connected component $G_0$ of $G_\sigma$, and $G_0$ is aperiodic. Let $\mathcal{A}_0 \subset \mathcal{A}$ denote the vertex set of $G_0$. Then the substitution obtained by restricting $\sigma$ to the domain $\mathcal{A}_0^+$ has irreducible aperiodic incidence matrix and is hence primitive. 

\begin{lemma}\label{lem: bounding edit distance between ess and iness generated words, unif}
	Let $\sigma$ be an $\ell$-uniform substitution such that $\ell > 1$ and there is a unique essential strongly connected component $G_0$ of $G_\sigma$, and $G_0$ is aperiodic. Let $\mathcal{A}_0 \subset \mathcal{A}$ denote the vertex set of $G_0$. Let $\tau$ denote the primitive substitution obtained by restricting $\sigma$ to the domain $\mathcal{A}_0^+$. Then
	\begin{equation}
		\max\{ \editfull(w,\tilde{w}) : w \in \mathcal{W}_n(\tau) \text{ and } \tilde{w} \in \mathcal{W}_n(\sigma) \setminus \mathcal{W}_n(\tau) \} \ = \ o(n).
	\end{equation}
\end{lemma}
\begin{proof}
	Fix $w \in \mathcal{W}_n(\tau)$ and $\tilde{w} \in \mathcal{W}_n(\sigma) \setminus \mathcal{W}_n(\tau)$. Let $m,\tilde{m} \in \N$, $a_0 \in \mathcal{A}_0$, and $\tilde{a}_0 \in \mathcal{A}\setminus \mathcal{A}_0$ be such that $w \prec \sigma^m(a_0)$ and $\tilde{w} \prec \sigma^{\tilde{m}}(\tilde{a}_0)$. Let $k \in \N \cup \{0\}$ be such that $\ell^{2k} \leq n < \ell^{2k+2}$; then $m \geq 2k$ and $\tilde{m} \geq 2k$. 
	
	Let $y$ be a maximal length subword of $\sigma^{m-k}(a_0)$ such that $\sigma^{k}(y) \prec w$. Then
	\begin{equation}\label{eqn 1 in lem: bounding edit distance between ess and iness generated words, unif}
		\frac{n}{\ell^k} - 2 \ \leq \ \len(y) \ \leq \ \frac{n}{\ell^k}.
	\end{equation}
	Similarly, let $\tilde{y}$ be a maximal length subword of $\sigma^{\tilde{m}-k}(\tilde{a}_0)$ such that $\sigma^{k}(\tilde{y})~\prec~\tilde{w}$. 
	
	Let $x$ (resp. $\tilde{x}$) consist of the first $q := \lceil n/\ell^k - 2\rceil$ letters of $y$ (resp. $\tilde{y}$). Then \eqref{eqn 1 in lem: bounding edit distance between ess and iness generated words, unif} holds with $\len(y)$ replaced by $q$. 
	
	Write $x = b_1 \cdots b_q$, where all $b_i \in \mathcal{A}_0$, and $\tilde{x} = \tilde{b}_1 \cdots \tilde{b}_{q}$, where $\tilde{b}_i \in \mathcal{A}$. For each $i \in \{1,\ldots, q\}$, put $B_i := \sigma^{k}(b_i)$ and $\tilde{B}_i := \sigma^{k}(\tilde{b}_i)$. Let $Y_1,Y_3, \tilde{Y}_1,\tilde{Y}_3$ be the (possibly empty) words of length at most $\ell^{k}$ such that
	\begin{equation}
		w \ = \ Y_1B_1 \cdots B_qY_3
	\end{equation}
	and
	\begin{equation}
		\tilde{w} \ = \ \tilde{Y}_1\tilde{B}_1 \cdots \tilde{B}_{q}\tilde{Y}_{3}.
	\end{equation}
	Then trivially estimate
	\begin{equation}
		\editfull(Y_1,\tilde{Y}_1) + \editfull(Y_3,\tilde{Y}_3) \ \leq \ 2\ell^{k}.
	\end{equation}
	
	Let $I \subset \{1,\ldots, q\}$ be defined by $i \in I$ if and only if $\tilde{b}_i \not\in \mathcal{A}_0$. For each $i \in I$, again trivially estimate
	\begin{equation}
		\editfull(B_i,\tilde{B}_i) \ \leq \ \ell^{k}.
	\end{equation}
	
	For each $i \in \{1,\ldots, q\} \setminus I$, observe the slightly less trivial estimate
	\begin{equation}
		\editfull(B_i,\tilde{B}_i) \ \leq \ \mathrm{diam}_E(\mathcal{W}_{\ell^{k}}(\tau)).
	\end{equation}
	Therefore,
	\begin{multline}\label{eqn 1.5 in lem: bounding edit distance between ess and iness generated words, unif}
		\editfull(w,\tilde{w}) \ \leq \ (2+|I|)\ell^{k} + (q - |I|)\mathrm{diam}_E(\mathcal{W}_{\ell^{k}}(\tau)) \\
		\leq \ (2+|I|)\ell^{k} + q\cdot \mathrm{diam}_E(\mathcal{W}_{\ell^{k}}(\tau)).
	\end{multline}
	Now, since $q \leq n/\ell^k$, it follows by Lemma~\ref{lem: unif case, ess letters dominate general words} that
	\begin{equation}
		\max_{\tilde{w} \in \mathcal{W}_n(\sigma) \setminus \mathcal{W}_n(\tau)} |I| \ = \  o(n/\ell^k),
	\end{equation}
	hence
	\begin{equation}\label{eqn 2 in lem: bounding edit distance between ess and iness generated words, unif} 
		\max_{\substack{w \in \mathcal{W}_n(\sigma), \\ \tilde{w} \in \mathcal{W}_n(\sigma) \setminus \mathcal{W}_n(\tau)}} (2+|I|)\ell^{k} \ = \ o(n).
	\end{equation}
	By Theorem~\ref{main theorem}, $\mathrm{diam}_E(\mathcal{W}_{\ell^{k}}(\tau)) = o(\ell^k)$. Since $q \leq n/\ell^k$,
	\begin{equation}\label{eqn 3 in lem: bounding edit distance between ess and iness generated words, unif}
		\max_{\substack{w \in \mathcal{W}_n(\sigma), \\ \tilde{w} \in \mathcal{W}_n(\sigma) \setminus \mathcal{W}_n(\tau)}} q\cdot \mathrm{diam}_E(\mathcal{W}_{\ell^{k}}(\tau)) \ = \ o(n).
	\end{equation}
	Combining \eqref{eqn 1.5 in lem: bounding edit distance between ess and iness generated words, unif}, \eqref{eqn 2 in lem: bounding edit distance between ess and iness generated words, unif}, and \eqref{eqn 3 in lem: bounding edit distance between ess and iness generated words, unif} proves the lemma.
\end{proof}

Now we are ready to prove the proposition, which we first restate: 
	\addtocounter{introtheorem}{-2}
	\begin{introproposition}
		Let $\sigma$ be an $\ell$-uniform substitution with $\ell > 1$.
		\begin{enumerate}[label=(\alph*)]
			\item If there is a unique essential strongly connected component $G_0$ of $G_\sigma$, and $G_0$ is aperiodic, then $\mathrm{diam}_E(\mathcal{W}_n(\sigma)) \ = \ o(n)$.
			\item Otherwise, $\mathrm{diam}_E(\mathcal{W}_n(\sigma)) \ = \ n$ for all $n \in \N$.
		\end{enumerate}
	\end{introproposition}	

\begin{proof}
	Suppose there is a unique essential strongly connected component $G_0$ of $G_\sigma$, and $G_0$ is aperiodic. Let $\tau$ denote the primitive substitution obtained by restricting $\sigma$ to the domain $\mathcal{A}_0^+$.
	By Theorem~\ref{main theorem},
	\begin{equation}
		\max_{\substack{w \in \mathcal{W}_n(\tau), \\ \tilde{w} \in \mathcal{W}_n(\tau)}} \editfull(w,\tilde{w}) \ = \ o(n).
	\end{equation}
	By Lemma~\ref{lem: bounding edit distance between ess and iness generated words, unif},
	\begin{equation}
		\max_{\substack{w \in \mathcal{W}_n(\tau), \\ \tilde{w} \in \mathcal{W}_n(\sigma) \setminus \mathcal{W}_n(\tau)}} \editfull(w,\tilde{w}) \ = \ o(n).
	\end{equation}
	By the triangle inequality, these together imply $\mathrm{diam}_E(\mathcal{W}_n(\sigma)) \ = \ o(n)$.
	
	Suppose the statement ``there is a unique essential strongly connected component $G_0$ of $G_\sigma$, and $G_0$ is aperiodic'' is false. Then, because the decomposition of $G_\sigma$ into strongly connected components must have an essential component, either there are at least two essential strongly connected components $G_0$ and $G_1$ of $G_\sigma$, or $G_0$ is not aperiodic.
	
	In the first case, let $a \in \mathcal{A}$ (resp. $b \in \mathcal{A}$) be a letter in the vertex set of $G_0$ (resp. $G_1$), and observe that, for all $k \in \N$,
	\begin{equation}
		\editfull(\sigma^k(a),\sigma^k(b)) \ = \ \ell^k
	\end{equation}
	since $\sigma^k(a)$ (resp. $\sigma^k(b)$) consists solely of letters from the vertex set of $G_0$ (resp. $G_1$).
	
	In the second case, that $G_0$ is an essential strongly connected component that is not aperiodic, let $p_0 > 1$ be the period. Then the vertex set of $G_0$ has more than one element. In particular, regarding the decomposition of $G_{\sigma^{p_0}}$ into strongly connected components, there will be at least two such components whose vertex sets are disjoint and contained in the vertex set of $G_0$. Taking a letter $a$ from one such component and a letter $b$ from another such component, we observe that, for all $k \in \N$,
	\begin{equation}
		\editfull(\sigma^{kp_0}(a),\sigma^{kp_0}(b)) \ = \ \ell^{kp_0}.
	\end{equation}
	
	In either case, we obtain
	\begin{equation}
		\mathrm{diam}_E(\mathcal{W}_n(\sigma)) \ = \ n
	\end{equation}
	for all $n \in \N$. 
\end{proof}

\section{Lower bound on edit distance for Thue--Morse}
\label{sec: proof of edit distance lower bound}

In this section, we prove Proposition~\ref{introprop: lower bound on edit distance for morse}.

The following lemma is due to Thue \cite{thue1,thue2}.
\begin{lemma}\label{lem: properties of mu}
	The Thue--Morse sequence $\mathbf x$ does not contain any subword of the form $www$ or $\xi w\xi w\xi $, where $w \in \{0,1\}^+$ and $\xi \in \{0,1\}$.
\end{lemma}
For the reader's convenience, we explain the special case we need, namely, that $\mathbf{x}$ does not contain the subwords $10101$ or $01010$. Let $s_2(n)$ be the number of $1$'s in the binary expansion of $n \in \N \cup \{0\}$. Then $s_2(0) = 0$, and for all $n \in \N \cup \{0\}$, one has $s_{2}(2n) = s_2(n)$ and $s_2(2n+1) = s_2(n)+1$. Hence, we have $x_n \equiv s_2(n) \bmod 2$ for all $n$, so it suffices to show that $s_2(i)$ for $i = n, n+2, n+4$ are not all of the same parity modulo 2. Indeed, observe that either $n \bmod 4$ or $n+2 \bmod 4$ belong to $\{0,1\}$. In the first case, $s_2(n+2) = s_2(n) + 1$, and in the second case,   $s_2(n+4) = s_2(n+2)+1$, as desired.

We now relate the length of a common subword of $\mathbf{x}$ and $\overline{\mathbf{x}}$ to the relative shift  between the locations it appears in each of them.

\begin{lemma}\label{lem: perfect match in TM requires shift at least 1/3 of word length, base}
	Let $1 \leq k \leq 4$ and $a \geq 0$. Suppose $w = x_{[a,a+k)} = \overline{x_{[b,b+k)}}$, where $b := a + s$ for some $s > 0$. Then $s \geq k/3$. 
\end{lemma}
\begin{proof}
	The cases $k = 1,2,3$ are trivial.
	
	When $k = 4$, the possible words $w$ are, without loss of generality, one of 0110, 0010, 0011, 0100, 0101. (The other five words are bitwise complements of these.) We must verify that $ s \ne 1$. This is clear if $00 \prec w$ or $11 \prec w$, so it remains to handle $w = 0101$. If $s=1$, then $1 = x_{a +3} = \overline{x_{a+4}}$, so $x_{a+4} = 0$ and thus $01010 \prec \mathbf{x}$, contradicting Lemma~\ref{lem: properties of mu}.
\end{proof}

\begin{lemma}\label{lem: perfect match in TM requires shift at least 1/3 of word length}
	Let $k \geq 1$ and $a \geq 0$. Suppose $ x_{[a,a+k)} = \overline{x_{[b,b+k)}}$, where $b := a + s$ for some $s \in \Z \setminus \{0\}$. Then $|s| \geq k/3$.
\end{lemma}
\begin{proof}
	We may suppose that $s > 0$, since if $x_{[a,a+k)} \ = \ \overline{x_{[b,b+k)}}$ for some $a>b$, then taking bitwise complements would switch the roles of $a$ and $b$.  
	
	We induct on $k$. Lemma~\ref{lem: perfect match in TM requires shift at least 1/3 of word length, base} handles the cases $1 \leq k \leq 4$, so we may assume that $k \geq 5$ and the conclusion has been shown for $k' \in [1,k)$. By Lemma~\ref{lem: properties of mu}, the word $w := x_{[a,a+k)}$ must contain 00 or 11 as a subword. For concreteness, suppose $w$ contains the subword $x_{a'}x_{a'+1} = 11$, where $a' \in \{a,\ldots, a+3\}$. By the recursive definition of $\mathbf{x}$, every subword 00 or 11 of $\mathbf{x}$ or $\overline{\mathbf{x}}$ must begin at an odd index, so $a'$ is odd and $s$ is even.
	
	If $a$ is odd, then $x_{a-1} = \overline{x_a} = x_b = \overline{x_{b-1}}$, so we may replace $a$ and $b$ by $a-1$ and $b-1$ respectively;  thus we may assume that $a$ and $b$ are even.
	Since
	\begin{equation} \label{definition of morse} 
		\forall n\in \N, \quad x_n \ = \ x_{2n},
	\end{equation}
	we  deduce that
	\begin{equation}
		x_{[a/2,a/2+k')} \ = \ \overline{x_{[b/2,b/2+k')}} \,,
	\end{equation}
	where $k' :=\lfloor(k+1)/2 \rfloor $. 
	By the inductive hypothesis   $s/2 \geq k'/3 \ge k/6$, and the inductive step is complete.
\end{proof}

Now we prove Proposition~\ref{introprop: lower bound on edit distance for morse}.

\begin{proof}[Proof of Proposition~\ref{introprop: lower bound on edit distance for morse}]
	Fix $n$. Let $\Delta := \editfull(x_{[0,n)},\overline{x_{[0,n)}})$ and $m := n - \Delta$. Fix strictly increasing maps $\pi,\tau : \{1,\ldots, m\} \to \{0,\ldots, n-1\}$ such that $x_{\pi(i)} = \overline{x_{\tau(i)}}$ for each $i$. Let
	\begin{equation}
		S_{\pi} \ := \ \bigl\{ i \in \{1,\ldots,m-1\} : \pi(i+1) > \pi(i) + 1\bigr\}
	\end{equation} 
	and define $S_{\tau}$  similarly. Note that $|S_{\pi}| \leq \Delta$, so  $S_* := S_{\pi} \cup S_{\tau}$ satisfies $|S_*| \leq 2\Delta$. Next, observe that the longest run in $\{1,\ldots,m-1\} \setminus S_*$ has length at least
	\begin{equation}
		\frac{m-1-|S_*|}{|S_*|+1} \ \geq \ \frac{m-1-2\Delta}{2\Delta + 1}.
	\end{equation}
	Thus, there exist integers $a,s,k$ with $a$ nonnegative and 
	\begin{equation}
		k \ \geq \ \frac{m-1-2\Delta}{2\Delta + 1} + 1 \ = \ \frac{n-\Delta}{2\Delta+1}
	\end{equation} 
	such that $x_{[a,a+k)} = \overline{x_{[a+s,a+s+k)}}$. By Lemma~\ref{lem: perfect match in TM requires shift at least 1/3 of word length}, we have $|s| \geq k/3$ and hence $t := \max_{1\leq i \leq m} |\pi(i)-\tau(i)|$ satisfies $t \geq k/3$.
	
	Then, since $\pi(i) - i$ is nondecreasing, we have, for all $i$, 
	\begin{equation}
		-1 \ \leq \ \pi(i) - i \ \leq \ \pi(m) - m \ \leq \ \Delta - 1
	\end{equation}
	and similarly for $\tau$, hence $\pi(i) - \tau(i) \leq \Delta$. Therefore, $t \leq \Delta$, hence
	\begin{equation}
		\Delta \ \geq \ \frac{n-\Delta}{6\Delta + 3},
	\end{equation}
	whence $\Delta \geq \sqrt{\frac{n}{6}}  -1$.
\end{proof}

\subsection{Open problem: Determining growth of edit distance}

By Proposition~\ref{introprop: lower bound on edit distance for morse} and \eqref{Blikstad bound on bb}, there exist constants $c,\tilde{C}$ such that
\begin{equation}\label{inequalities on edit distance of length n prefix}
	c\sqrt{n}\ \leq \ \editfull(x_{[0,n)},\overline{x_{[0,n)}}) \ \leq \ \tilde{C} \frac{n\sqrt{\log n}}{2^{\beta\sqrt{\log n}} }.
\end{equation}
It remains to compute the true growth rate. Hence one may first ask, what are the values of 
\begin{equation}
	\limsup_{n \to \infty} \frac{\log \editfull\bigl(x_{[0,n)},\overline{x_{[0,n)}}\bigr)}{\log n}  \ \in \ [1/2,1]
\end{equation}
and the corresponding $\liminf$? In particular, are these equal?

The sequence $\bb{n} := n - \editfull(x_{[0,n)},\overline{x_{[0,n)}})$ is discussed  in the OEIS~\cite{oeis}.

More generally, one can ask: Given a primitive substitution $\sigma$, does
\begin{equation}
	\lim_{n \to \infty} \frac{\log \mathrm{diam}_E(\mathcal{W}_n(\sigma))}{\log n}
\end{equation}
exist?
\\ \medskip \\
\noindent \textbf{Acknowledgment.} \quad We are grateful to a referee for useful references.

\end{document}